\documentclass[12pt]{article}

\usepackage{amsmath,amssymb,amsthm}
\usepackage{tikz}

\def\dfrac#1#2{\lower0.15ex\hbox{\large$\frac{#1}{#2}$}}

\newtheorem{theorem}{Theorem}[section]
\newtheorem{lemma}[theorem]{Lemma}

\newtheorem{corollary}[theorem]{Corollary}
\numberwithin{equation}{section}

\def\abs#1{\mathopen|#1\mathclose|} 
\def\kvec{{\boldsymbol{k}}}
\def\kmax{k_{\mathrm{max}}}
\def\gvec{{\boldsymbol{g}}}
\def\gmax{g_{\mathrm{max}}}
\def\lvec{{\boldsymbol{\ell}}}
\def\lmax{\ell_{\mathrm{max}}}
\def\Hrk{{\mathcal{H}_r(\kvec)}}
\def\Lrk{{\mathcal{L}_r(\kvec)}}
\def\Brk{{\mathcal{B}_r(\kvec)}}
\def\Brkzero{{\mathcal{B}^{(0)}_r(\kvec)}}
\def\Brkplus{{\mathcal{B}^{+}_r(\kvec)}}

\def\dist{\operatorname{dist}}

\def\dfrac#1#2{\lower0.15ex\hbox{\large$\frac{#1}{#2}$}}

\title{Asymptotic enumeration of\\ sparse uniform linear hypergraphs with given degrees}

\author{
Vladimir Blinovsky\\
\small Instituto de Matem{\' a}tica e Estat{\' \i}stica\\[-0.8ex]
\small Universidade de S{\~ a}o Paulo, 05508-090, Brazil\\[-0.8ex]
\small Institute for Information Transmission Problems\\[-0.8ex]
\small Russian Academy of Sciences, Moscow 127994, Russia\\
\small \tt vblinovs@yandex.ru\\
\and
Catherine Greenhill
\thanks{Research supported by the Australian Research Council Discovery Project DP140101519.} \\
\small School of Mathematics and Statistics\\[-0.8ex]
\small UNSW Australia\\[-0.8ex]
\small Sydney NSW 2052, Australia\\
\small \tt c.greenhill@unsw.edu.au\\
}

\date{13 July 2016}

\begin{document}

\maketitle

\vspace*{-1\baselineskip}

\begin{abstract}
A hypergraph is \emph{simple} if it has no loops and
no repeated edges, and a hypergraph is \emph{linear}  
if it is simple and each pair of edges intersects in at most one vertex. 
For $n\geq 3$, let $r= r(n)\geq 3$ be an integer and let
$\kvec = (k_1,\ldots, k_n)$ be a vector of nonnegative
integers, where each $k_j = k_j(n)$ may depend on $n$.
Let $M = M(n) = \sum_{j=1}^n k_j$ for all $n\geq 3$, and define
the set $\mathcal{I} = \{ n\geq 3 \mid r(n) \text{ divides } M(n)\}$.
We assume that $\mathcal{I}$ is infinite, and perform asymptotics
as $n$ tends to infinity along $\mathcal{I}$. 
Our main result is an asymptotic enumeration formula for linear
$r$-uniform hypergraphs with degree sequence
$\kvec$.  This formula holds whenever the maximum degree $\kmax$
satisfies $r^4\kmax^4(\kmax + r) = o(M)$.
Our approach is to work with the incidence matrix of a hypergraph, 
interpreted as the biadjacency matrix of a bipartite graph, enabling
us to apply known enumeration results for bipartite graphs. 
This approach also leads to a new asymptotic enumeration
formula for simple uniform hypergraphs with specified degrees,
and a result regarding the girth of random bipartite graphs
with specified degrees.
\end{abstract}

\vspace*{1\baselineskip}

\section{Introduction}\label{s:introduction}

Hypergraphs are combinatorial structures which can model
very general relational systems, including some real-world 
networks~\cite{ER,GZCN,Dih}.
Formally, a  \emph{hypergraph} or set system is defined as a pair
$(V,E)$, where $V$ is a finite set and 
$E$ is a multiset of multisubsets of $V$. 
(We refer to elements of $E$ as \emph{edges}.)
Note that under this definition, a hypergraph may contain repeated
edges and an edge may contain repeated vertices.

Any 2-element multisubset of an edge $e\in E$ is called a \emph{link}
in $e$.
If a vertex $v$ has multiplicity at least 2 in the edge $e$,
we say that $v$ is a \emph{loop} in $e$. 
(So every loop in $e$ is also a link in $e$.)
The \emph{multiplicity} of a link $\{ x,y\}$ is the number of
edges in $E$ which contain $\{ x,y\}$ (counting multiplicities).
A \emph{double link} is a link with multiplicity 2.

A hypergraph is 
\emph{simple} if it has no loops and no repeated edges: that is,
$E$ is a set of edges, and each edge is a subset of $V$.  Here
it is possible that distinct edges may have more than one vertex
in common.   (This definition of simple hypergraph
appears to be standard, and
matches the definition of simple hypergraphs given
by Berge~\cite{Berge} in the case of uniform hypergraphs.)
A hypergraph is called \emph{linear} if it
has no loops and each pair
of distinct edges intersect in at most one vertex. 
(Note that linear hypergraphs are also simple, when $r\geq 2$.)
Linear hypergraphs have been well-studied in many contexts
(sometimes they have been referred to as ``simple hypergraphs'').  
See for example~\cite{CR,FM,KK,NRSS}.

For a positive integer $r$, 
the hypergraph $(V,E)$ is $r$-\emph{uniform} 
if each edge $e\in E$ contains exactly $r$ vertices (counting multiplicities).
Uniform hypergraphs are a particular focus of study, not least
because a 2-uniform hypergraph is precisely a graph.  
We seek an 
asymptotic enumeration formula for the number of linear $r$-uniform
hypergraphs with a given degree sequence, when the maximum degree
is not too large (the sparse range), and allowing $r$ to grow slowly with $n$.

To state our result precisely, we need some definitions.
Write $[a] = \{ 1,2,\ldots, a\}$ for all positive integers $a$.
Given nonnegative integers $a$, $b$, let $(a)_b$ denote the falling factorial 
$a(a-1)\cdots (a-b+1)$.
We are given a degree sequence $\kvec = \kvec(n) = (k_{1},\ldots, k_{n})$  with sum
$M = M(n) = \sum_{i=1}^n k_i$,  and we are also given an
integer $r=r(n)\geq 3$, for each $n\geq 3$. 
Let $\kmax =\kmax(n) = \max_{j=1}^n\, k_j$ for all $n\geq 3$.
For each positive integer $t$, define 
\[ M_t = M_t(n) = \sum_{i=1}^n (k_i)_t.\]
Then $M_1=M$ and $M_t\leq \kmax\, M_{t-1}$ for $t\geq 2$.

Let $\Hrk$ denote the set of simple $r$-uniform hypergraphs on the
vertex set $[n]$ with degree sequence given by $\kvec = (k_1,\ldots, k_n)$,
and let $\Lrk$ be the set of all linear hypergraphs in $\Hrk$. 
Note that $\Hrk$ and $\Lrk$ are both empty unless $r$ divides $M$.
Our main theorem is the following.

\begin{theorem}
For $n\geq 3$, let $r= r(n)\geq 3$ be an integer and let
$\kvec = (k_1,\ldots, k_n)$ be a vector of nonnegative
integers, where each $k_j = k_j(n)$ may depend on $n$.
Let $\kmax =\kmax(n) = \max_{j=1}^n\, k_j$ for all $n\geq 3$.
Define $M = M(n) = \sum_{j=1}^n k_j$ for all $n\geq 3$, and suppose that
the set 
\[ \mathcal{I} = \{ n\geq 3 \mid r(n) \text{ divides } M(n)\}\]
is infinite.  
Suppose that $M\to\infty$ and $r^4\kmax^4(\kmax + r) = o(M)$ as $n$ tends to infinity 
along elements of $\mathcal{I}$. Then
\begin{align*}
 & |\Lrk|\\ &= 
   \frac{M!}{\left(M/r\right)!\, (r!)^{M/r}\, \prod_{i=1}^{n}\, k_i!}\,
 \exp\left( 
    {} - \frac{(r-1) M_2}{2M} 
  - \frac{(r-1)^2 M_2^2}{4M^2} 
  +  O\left(\frac{r^4\kmax^4(\kmax + r)}{M}\right) \right).
\end{align*}
\label{main}
\end{theorem}

We believe that Theorem~\ref{main} is the first asymptotic enumeration
by degree sequence for $r$-uniform linear hypergraphs with $r\geq 3$, 
and the first asymptotic enumeration result for sparse hypergraphs 
which allows $r$ to grow with $n$.
The two (non-error) terms within the exponential arise naturally: the first corresponds
to the expected number of loops and the second corresponds to the expected
number of double links.

A brief survey of the relevant literature is given in the next subsection.  
Note that when $r=2$ (graphs), our result is weaker than the formula given by
McKay and Wormald~\cite{McKW91}, as 
their expression has smaller error term and more significant terms.
In order to improve the accuracy of Theorem~\ref{main} to a similar level,
a more detailed analysis of double links is required.
We will present such an analysis in a future paper.  

To obtain Theorem~\ref{main}, we treat the incidence matrix of a hypergraph as the
biadjacency matrix of a bipartite graph, thereby enabling us
to make use of prior enumeration results for bipartite graphs
in order to enumerate linear hypergraphs.
In Section~\ref{s:simple}
we show that some undesirable substructures are rare in random 
bipartite graphs with the appropriate degrees.
As a corollary of this, we obtain a new enumeration result for
sparse simple uniform hypergraphs.  
Theorem~\ref{main} then follows from a switching
argument for bipartite graphs which is used to remove 4-cycles,
as these correspond to double links in the hypergraph.
This switching argument is presented in Section~\ref{s:doubles},
leading to the proof of Theorem~\ref{main}.  Finally in
Corollary~\ref{4cycle} we state a consequence of our calculations
relating to the girth of bipartite graphs with specified degrees. 

Using this approach of translating the problem to one involving
bipartite graphs, it should be possible to relax the uniformity condition,
perhaps by 
allowing the number of edges with a given size to be specified
up to a maximum edge size (which may grow modestly with $n$).
Such a generalisation has not been performed here.

\subsection{History }\label{s:history}

In the case of graphs, the best asymptotic formula in the sparse
range is given by McKay and Wormald~\cite{McKW91}.  See that paper 
for further history of the problem.  The dense range was
treated in~\cite{ranX,MW90}, but there is a gap between these
two ranges in which nothing is known.

An early result in the asymptotic enumeration of hypergraphs
was given by Cooper et al.~\cite{CFMR}, who considered simple
$k$-regular hypergraphs when $k=O(1)$.  More recently,
Dudek et al.~\cite{DFRS} proved an asymptotic formula for 
simple $k$-regular hypergraphs with $k=o(n^{1/2})$.
In~\cite{BG} this was extended to irregular sequences, 
with an improved error bound.  We restate this result below.

\begin{theorem}
\label{th:BG}
\emph{\cite[Theorem 1.1]{BG}}\,
Let $r\geq 3$ be a fixed integer. 
Let $\kvec$, $M$ and $\kmax$ be defined as in Theorem~\ref{main}.
Assume that $r$ divides $M$ for infinitely many values of $n$.
Suppose that $M\to\infty$, $\kmax \geq 2$ and $\kmax^{3} = o(M)$
as $n$ tends to infinity along these values. 
Then
\[ |\mathcal{H}_r(\kvec)| = \frac{M!}{(M/r)!\, (r!)^{M/r}\,
 \prod_{i=1}^n k_i!}\,
   \exp\left( -\frac{(r-1)M_2}{2M} + O(\kmax^{3}/M)\right).
\]
\end{theorem}

Kuperberg, Lovett and Peled~\cite{KLP}
gave an asymptotic formula for the number of dense simple $r$-uniform
hypergraphs with a given degree sequence.

\section{Hypergraphs, incidence matrices and bipartite graphs}\label{s:simple}

Suppose that $G$ is an $r$-uniform hypergraph with degree sequence $\kvec$ 
which has no loops (but may have repeated edges).
Let $A$ be the $n\times (M/r)$ incidence matrix of $G$,
where the rows of the incidence matrix correspond to vertices $1,2,\ldots n$
in that order, and the columns correspond to the edges of the hypergraph, in some order.   
Then $A$ is a 0-1 matrix (as $G$ has
no loops), the row sums of $A$ are given by $\kvec$ and each column
sum of $A$ equals $r$. 

If $G$ is simple (that is, if $G\in\Hrk$) then all columns of $A$ are distinct, 
and hence there
are precisely $(M/r)!$ possible (distinct) incidence matrices
corresponding to $G$.   
Conversely, every 0-1 matrix with rows sums given by $\kvec$,
column sums all equal to $r$ and with no repeated columns can be interpreted
as the incidence matrix of a hypergraph in $\Hrk$.

It will be convenient to work with the bipartite graphs whose
biadjacency matrices are the incidence matrices of hypergraphs.
Let $\Brk$ be the set of bipartite graphs
with vertex bipartition $\{ v_1,\ldots, v_n\}\cup\{ e_1, e_2,\ldots, e_{M/r}\}$,
such that degree sequence of $(v_1,\ldots, v_n)$ is $\kvec$ and
every vertex $e_j$ has degree $r$.
We sometimes say that a vertex $v_j$ is ``on the left'' and that a vertex $e_i$
is ``on the right''.
An example of a 3-uniform hypergraph, its incidence matrix (with
edges ordered in lexicographical order) and corresponding bipartite
graph is shown in Figure~\ref{f:example}.

\begin{figure}[ht]
\begin{center}
\begin{tikzpicture}
\draw[fill] (-0.5,3) circle (0.1);
\node[above] at (-0.5,3.1) {$1$};
\draw[fill] (0.5,3) circle (0.1);
\node[above] at (0.5,3.1) {$2$};
\draw[fill] (3.4,3) circle (0.1);
\node[above] at (3.4,3.1) {$3$};
\draw[fill] (-0.5,1.5) circle (0.1);
\node[below] at (-0.5,1.4) {$4$};
\draw[fill] (1.8,1.5) circle (0.1);
\node[below] at (1.8,1.4) {$5$};
\draw[fill] (2.7,1.5) circle (0.1);
\node[below] at (2.7,1.4) {$6$};
\draw [thick,rounded corners] (-1.2,0.6) rectangle (1.5,4.2);
\draw [thick,rounded corners] (-1.0,2.5) rectangle (3.8,4);
\draw [thick,rounded corners] (-1.0,0.4) rectangle (3.8,2);
\draw [thick,rounded corners] (0.15,3.8) rectangle (3.1,0.8);
\node [left] at (7.7,2.5) {$\begin{pmatrix} 1 & 1 & 0 & 0 \\ 1 & 1 & 1 & 0\\ 1 & 0 & 0 & 0\\ 0 & 1 & 0 & 1\\ 0 & 0 & 1 & 1\\ 0 & 0 & 1 & 1\end{pmatrix}$};
\draw[fill] (9.5,5) circle (0.1);
\node[left] at (9.4,5.0) {$v_1$};
\draw[fill] (9.5,4) circle (0.1);
\node[left] at (9.4,4.0) {$v_2$};
\draw[fill] (9.5,3) circle (0.1);
\node[left] at (9.4,3.0) {$v_3$};
\draw[fill] (9.5,2) circle (0.1);
\node[left] at (9.4,2.0) {$v_4$};
\draw[fill] (9.5,1) circle (0.1);
\node[left] at (9.4,1.0) {$v_5$};
\draw[fill] (9.5,0) circle (0.1);
\node[left] at (9.4,0.0) {$v_6$};
\draw[fill] (12.5,0.5) circle (0.1);
\node[right] at (12.6,0.5) {$e_4$};
\draw[fill] (12.5,1.8) circle (0.1);
\node[right] at (12.6,1.8) {$e_3$};
\draw[fill] (12.5,3.2) circle (0.1);
\node[right] at (12.6,3.2) {$e_2$};
\draw[fill] (12.5,4.5) circle (0.1);
\node[right] at (12.6,4.5) {$e_1$};
\draw [-,thick] (9.5,5) -- (12.5,4.5) -- (9.5,4) -- (12.5,3.2) -- (9.5,5);
\draw [-,thick] (9.5,3) -- (12.5,4.5);
\draw [-,thick] (9.5,4) -- (12.5,1.8) -- (9.5,1) -- (12.5,0.5) -- (9.5,0) -- (12.5,1.8) -- (9.5,1);
\draw [-,thick] (12.5,3.2) -- (9.5,2) -- (12.5,0.5);
\end{tikzpicture}
\caption{A hypergraph, its incidence matrix and corresponding bipartite graph.}
\label{f:example}
\end{center}
\end{figure}
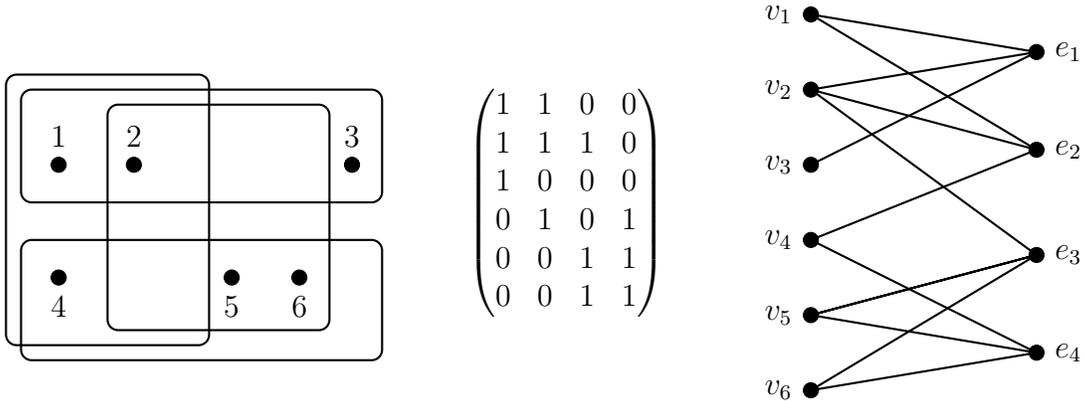

Double links will be  of particular interest: there are two double
links in the hypergraph in Figure~\ref{f:example}, and each corresponds
to a subgraph of the bipartite graph which is isomorphic to
$K_{2,2}$.  (One is induced by $\{ v_1,v_2, e_1, e_2\}$ and the other
by $\{ v_5,v_6,e_3,e_4\}$.)

It follows from~\cite[Theorem 1.3]{GMW} that
\begin{align}
\label{base}
 |\Brk| = \frac{M!}{(r!)^{M/r}\, \prod_{j=1}^n k_j!}\,
   \exp\left(-\frac{(r-1)M_2}{2M} + O(r^2\kmax^2/M)\right)
\end{align}
whenever $1\leq r\kmax = o(M^{1/2})$.
(In fact, the result of~\cite{GMW} is more accurate but we are
unable to exploit the extra accuracy here, so we state a simplified version.)

Next, let $\Brkzero$ denote the set of all bipartite graphs in
$\Brk$ such that no two vertices $e_{i_1}$, $e_{i_2}$
(on the right)
have the same neighbourhood.  These bipartite graphs correspond
to 0-1 matrices with no repeated columns, which in turn can be
viewed as incidence matrices of (simple) hypergraphs in $\Hrk$.
Hence
\begin{equation}
\label{step1}
 (M/r)!\, |\Hrk| = |\Brkzero|.
\end{equation}
To work towards linear hypergraphs, we identify some desirable properties
of the corresponding bipartite graphs.
Given an ordered pair of positive integers $(a, b)$, 
say that the bipartite graph $B$ has a 
\emph{copy of $K_{a,b}$} if $B$ contains a set of $a$ vertices on the left and
$b$ vertices on the right which induce a subgraph isomorphic to $K_{a,b}$.
This definition is slightly non-standard, since it is not symmetric with
respect to $a$ and $b$. However, we will mostly be interested in copies of
$K_{2,2}$, in which case there is no asymmetry.  We will write ``$B$ contains
a 4-cycle'' rather than ``$B$ contains a copy of $K_{2,2}$''.

Define
\begin{align*}
  N_2 &= 3\, \max\left\{ \lceil \, \log M\rceil,\, \lceil 2(r-1)^2 M_2^2/M^2\rceil\right\}
\end{align*}
and let $\Brkplus$ denote the set of all bipartite
graphs $B\in \Brk$ which satisfy the following properties:
\begin{enumerate}
\item[(i)]  $B$ has no copy of $K_{3,2}$ (with three vertices on the left and two on the right).
\item[(ii)]  $B$ has no copy of $K_{2,3}$ (with two vertices on the left and three on the right).
\item[(iii)] No two 4-cycles in $B$ have a vertex $e_j$ 
(on the right)
in common.  (This implies that any 4-cycles in $B$ are edge-disjoint.)
\item[(iv)] Any three distinct 4-cycles in $B$ involve at least five vertices
on the left.  (Together with (iii), this implies that if two distinct 4-cycles 
share a vertex on the left then any other 4-cycle in $B$ must be 
vertex-disjoint from the first two.)
\item[(v)] The number of 4-cycles in $B$ is at most $N_2$. 
\end{enumerate}
To motivate this definition, note that $B\in\Brkplus$ if
and only if 
the corresponding hypergraph $G=G(B)$ satisfies the following properties:
\begin{enumerate}
\item[(i)$'$]  The intersection of any two edges of $G$ contains 
at most two vertices.
\item[(ii)$'$]  Any link has multiplicity at most two in $G$.  
(That is, the intersection
of any three edges of $G$ contains at most one vertex.)
\item[(iii)$'$] No edge of $G$ contains more than one double link.  (That is,
if $e_1$ and $e_2$ are edges of $G$ which
share a double link then $e_1$ is not involved in any
other double link in $G$, and similarly for $e_2$.)
\item[(iv)$'$]  No vertex can belong to three double links, and if a
vertex $v$ belongs to two double links (say $\{ v,x\}$ and $\{ v,y\}$ are
both double links) then both $x$ and $y$ belong to precisely one double link.
\item[(v)$'$] There are at most $N_2$ double links in $G$.
\end{enumerate}
In particular, as $r\geq 3$, any hypergraph $G=G(B)$ with $B\in\Brkplus$
has (no loops and)
no repeated edges, so is simple.

McKay~\cite{McKay81} proved asymptotic formulae
for the probability that a randomly chosen bipartite graph with specified
degrees contains a fixed subgraph, under certain conditions.
We state one of these results below, which will be use repeatedly.
(In fact the statement below is a special case of~\cite[Theorem 3.5(a)]{McKay81}, obtained by taking $J=L$ and $H=\emptyset$ in the notation of~\cite{McKay81},
and with slightly simplified notation.)

\begin{lemma} \emph{(\cite[Theorem 3.5(a)]{McKay81})}\
Let $\mathcal{B}(\gvec)$ denote the set of bipartite graphs with 
vertex bipartition given by $\{ a_1,\ldots, a_n\}\cup\{ b_1,\ldots, b_m\}$
and degree sequence 
\[ \gvec = (g_1,\ldots, g_n;g_{1}',\ldots, g_{m}').\] 
(Here vertex $a_i$ has degree $g_i$ for $i=1,\ldots, n$, and vertex $b_j$ has
degree $g_{j}'$ for $j=1,\ldots, m$.)
Let $L$ be a subgraph of the complete
bipartite graph on this vertex bipartition, and let $\mathcal{B}(\gvec,L)$ be
the set of bipartite graphs in $\mathcal{B}(\gvec)$ which contain $L$
as a subgraph.  Write $E_{\gvec} = \sum_{i=1}^n g_i$ and 
$E_{\lvec} = \sum_{i=1}^n \ell_i$, where $\lvec = (\ell_1,\ldots, \ell_n;\ell'_1,\ldots, \ell'_m)$ 
is the degree sequence of $L$.  Finally, let $\gmax$ and $\lmax$
denote the maximum degree in $\gvec$ and $\lvec$, respectively,
and define
\[ \Gamma = 2\gmax(\gmax + \lmax-1) + 2.\]  
If $E_{\gvec} - \Gamma \geq E_{\lvec}$ then
\[ \frac{|\mathcal{B}(\gvec,L)|}{|\mathcal{B}(\gvec)|} \leq
    \frac{ \prod_{i=1}^n (g_i)_{\ell_i}\, \prod_{j=1}^m (g_j')_{\ell_j'}}{(E_{\gvec} - \Gamma)_{E_{\lvec}}}.
\]
\label{mckay}
\end{lemma}

Using this lemma, we now analyse the probability that
a uniformly random element of $\Brk$ satisfies
properties (i)--(v). 

\begin{theorem}
Under the conditions of Theorem~\ref{main}, 
\[ \frac{|\Brkplus|}{|\Brk|} = 1 + O\left(r^5\kmax^4/M\right).
\]
\label{easy}
\end{theorem}

\begin{proof}
Throughout this proof, consider a uniformly random element $B\in\Brk$.
We will apply Lemma~\ref{mckay} several times 
with $\gvec = (k_1,\ldots, k_n; r,\ldots, r)$.  
In each application, $L$ is a subgraph with constant maximum degree.
Hence
$\gmax = \max\{\kmax,r\}$ and
\[ \Gamma = 2\gmax(\gmax + \lmax-1) + 2 = O(r^2+\kmax^2).\]

For (i), let $v_{j_1}$, $v_{j_2}$, $v_{j_3} \in [n]$ be distinct vertices 
on the left, and let $e_{i_1}, e_{i_2}$ be distinct vertices on the right.
Applying Lemma~\ref{mckay} with $L = K_{3,2}$, we find that the 
probability that $B$ has a copy of $K_{3,2}$ on the vertices
$\{ v_{j_1}, v_{j_2}, v_{j_3}\}\cup\{e_{i_1},e_{i_1}\}$ is 
at most
\[ \frac{r^2(r-1)^2(r-2)^2}{(M + O(r^2+\kmax^2))_{6}}\,  (k_{j_1})_2\, (k_{j_2})_2\,
        (k_{j_3})_2.
\]
By assumption, $\kmax^2 +r^2= o(M)$.
Multiplying this by the number of choices for $\{ e_{i_1},e_{i_2}\}$
 and summing over all choices of $(j_1,j_2, j_3)$ with $1\leq j_1 < j_2
< j_3\leq n$ shows that the expected number of copies of $K_{3,2}$ in $B$ is at most
\begin{align}
 \binom{M/r}{2}\, \sum_{j_1 < j_2 < j_3} \,
    (k_{j_1})_2 \, (k_{j_2})_2\, (k_{j_3})_2\,\,
   O\left(\frac{r^6}{M^{6}}\right)\,
  &= \, O\left(\frac{r^4\, M_2^3}{M^{4}}\right) \notag\\
  &= \, O(r^4\kmax^3/M).
\label{eq1} 
\end{align}
Hence property (i) fails with probability $O(r^4\, \kmax^3/M)$.
For future reference, we note that the argument leading to (\ref{eq1})
still holds under the weaker condition $r^4\kmax^3= o(M)$ (as
this condition still implies that $\kmax^2 + r^2 = o(M)$, and
all other calculations are unchanged).

Repeating this argument with $L = K_{2,3}$  shows that property (ii)
fails with probability $O(r^3\kmax^4/M)$. Using the subgraphs
$L$ shown in Figure~\ref{fig:rare} (a) and (b)
we can establish that 
property (iii) fails with probability $O(r^5\, \kmax^4/M)$.
\begin{figure}[ht!]
\begin{center}
\begin{tikzpicture}
\draw [fill] (1,1) circle (0.1);
\draw [fill] (1,2) circle (0.1);
\draw [fill] (1,3) circle (0.1);
\draw [fill] (2,1) circle (0.1);
\draw [fill] (2,2) circle (0.1);
\draw [fill] (2,3) circle (0.1);
\draw [-,thick] (2,2) -- (1,1) -- (2,1) -- (1,2) -- (2,1);
\draw [-,thick] (1,2) -- (2,3) -- (1,3) -- (2,2) -- (1,2);
\node [below] at (1.5,0) {(a)};
\draw [fill] (4,0.8) circle (0.1);
\draw [fill] (4,1.8) circle (0.1);
\draw [fill] (4,2.2) circle (0.1);
\draw [fill] (4,3.2) circle (0.1);
\draw [fill] (5,1) circle (0.1);
\draw [fill] (5,2) circle (0.1);
\draw [fill] (5,3) circle (0.1);
\draw [-,thick] (4,0.8) -- (5,1) -- (4,1.8) -- (5,2) -- (4,0.8);
\draw [-,thick] (4,2.2) -- (5,2) -- (4,3.2) -- (5,3) -- (4,2.2);
\node [below] at (4.5,0) {(b)};
\draw [fill] (7,1.0) circle (0.1);
\draw [fill] (7,2.0) circle (0.1);
\draw [fill] (7,3.0) circle (0.1);
\draw [fill] (8,0) circle (0.1);
\draw [fill] (8,1) circle (0.1);
\draw [fill] (8,1.5) circle (0.1);
\draw [fill] (8,2.5) circle (0.1);
\draw [fill] (8,3) circle (0.1);
\draw [fill] (8,4) circle (0.1);
\draw [-,thick] (7,1.0) -- (8,0) -- (7,2.0) -- (8,1) -- (7,1.0);
\draw [-,thick] (7,1.0) -- (8,1.5) -- (7,3.0) -- (8,2.5) -- (7,1.0);
\draw [-,thick] (7,2.0) -- (8,3) -- (7,3.0) -- (8,4) -- (7,2.0);
\node [below] at (7.5,0) {(c)};
\draw [fill] (10,0.5) circle (0.1);
\draw [fill] (10,1.5) circle (0.1);
\draw [fill] (10,2.5) circle (0.1);
\draw [fill] (10,3.5) circle (0.1);
\draw [fill] (11,0) circle (0.1);
\draw [fill] (11,1) circle (0.1);
\draw [fill] (11,1.5) circle (0.1);
\draw [fill] (11,2.5) circle (0.1);
\draw [fill] (11,3) circle (0.1);
\draw [fill] (11,4) circle (0.1);
\draw [-,thick] (10,0.5) -- (11,0) -- (10,1.5) -- (11,1) -- (10,0.5);
\draw [-,thick] (10,1.5) -- (11,1.5) -- (10,2.5) -- (11,2.5) -- (10,1.5);
\draw [-,thick] (10,3.5) -- (11,3) -- (10,2.5) -- (11,4) -- (10,3.5);
\node [below] at (10.5,0) {(d)};
\draw [fill] (13,0.5) circle (0.1);
\draw [fill] (13,1.5) circle (0.1);
\draw [fill] (13,2.5) circle (0.1);
\draw [fill] (13,3.5) circle (0.1);
\draw [fill] (14,0) circle (0.1);
\draw [fill] (14,1) circle (0.1);
\draw [fill] (14,1.5) circle (0.1);
\draw [fill] (14,2.5) circle (0.1);
\draw [fill] (14,3) circle (0.1);
\draw [fill] (14,4) circle (0.1);
\draw [-,thick] (13,0.5) -- (14,0) -- (13,1.5) -- (14,1) -- (13,0.5);
\draw [-,thick] (13,1.5) -- (14,1.5) -- (13,2.5) -- (14,2.5) -- (13,1.5);
\draw [-,thick] (13,1.5) -- (14,3) -- (13,3.5) -- (14,4) -- (13,1.5);
\node [below] at (13.5,0) {(e)};
\end{tikzpicture}
\caption{Rare subgraphs}
\label{fig:rare}
\end{center}
\end{figure}
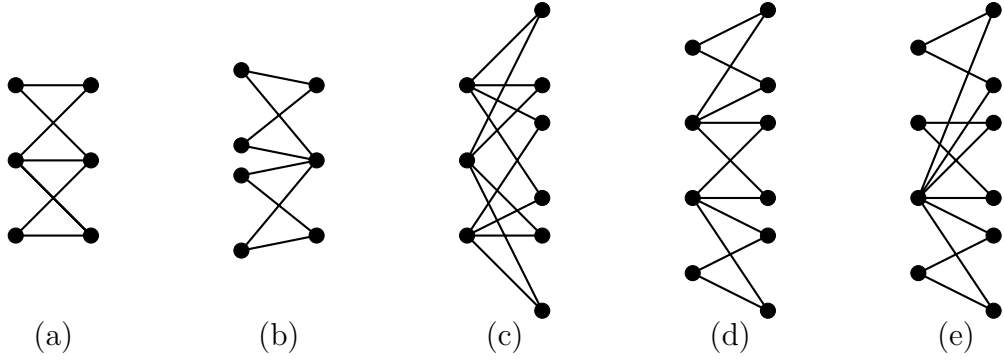
Assuming that properties (i)--(iii) hold,  we can prove
that property (iv) holds with probability $O(r^2\kmax^3/M)$
by considering the subgraphs $L$ shown in 
Figure~\ref{fig:rare} (c), (d) and (e). 

Now we turn to (v). 
Let
$Q_1 = \max\left\{ \lceil\log M\rceil,\, \lceil 2(r-1)^2 M_2^2/M^2\rceil \right\}$
and define $d=Q_1+1$.
We first show that the expected number of sets of $d$ vertex-disjoint
4-cycles in $B$ is $O(1/M)$.
Fix $(j_1,\ldots, j_{2d})\in [n]^{2d}$ such that $k_{j_\ell}\geq 2$ for 
$\ell=1,2,\ldots, 2d$
and $j_{2\ell-1}\neq j_{2\ell}$ for $\ell=1,2,\ldots, d$.
Let $(i_1,\ldots, i_{2d})\in \{ 1,\ldots, M/r\}^{2d}$ be a $(2d)$-tuple of 
(distinct) edge
labels.  The probability that there is a 4-cycle on 
$\{ v_{j_{2\ell-1}},v_{j_{2\ell}}\}\cup\{ e_{i_{2\ell-1}},e_{i_{2\ell}}\}$
for $\ell=1,\ldots, d$ is
\[ \prod_{\ell=1}^{2d} (k_{j_\ell})_2\,\, O\left( \frac{(r(r-1))^{2d}}{M^{4d}}\right),
\]
by~Lemma~\ref{mckay}.
There are at most $(M/r)^{2d}$ choices for $(i_1,\ldots, i_{2d})$, and for 
an upper bound we can sum over
all possible values of $(j_1,\ldots, j_{2d})$. 
This counts each set of $d$ vertex-disjoint 4-cycles precisely
$4^d\, d!$ times.
It follows that the expected number
of sets of $d$ vertex-disjoint 4-cycles in $B$ is
\begin{align*}
  \sum_{(j_1,\ldots, j_{2d})\in [n]^{2d}}\, 
 \prod_{\ell=1}^{2d} (k_{j_\ell})_2\,\, 
   O\left( \frac{(r-1)^{2d}}{4^d\, d!\, M^{2d}}\right)
 &=
  O\left(\frac{1}{d!}\, \left(\frac{(r-1)^{2}\, M_2^{2}}{4\,M^{2}}\right)^d\, \right)\notag \\
  &= O\left(\left(\frac{e\, (r-1)^2\, M_2^2}{4d\,M^2}\right)^d\, \right) \notag\\
  &= O\left((e/8)^d\, \right) \notag \\
  &= O(1/M) 
\end{align*}
by choice of $d$.
Next, let 
$Q_2 = \max\left\{ \lceil \log M \rceil,\, \lceil (r-1)^4 M_2^2 M_4/M^4\rceil \right\}$
and define $b=Q_2 + 1$. 
Assuming that properties (iii) and (iv) hold,
any 4-cycle in $B$ is
either vertex-disjoint from all other 4-cycles in $B$,
or shares one vertex on the left with precisely one other 4-cycle
in $B$. In the latter case, call such a pair of 4-cycles a \emph{fused pair}.  
Arguing as above, the expected number of sets of $b$ fused pairs
is at most
$O(1/M)$, by choice of $b$.
It follows that with probability $1 + O(r^5\kmax^4/M)$,
the number of 4-cycles
in $B$ is at most $Q_1 + 2Q_2\leq 3Q_1 = N_2$,
completing the proof.
\end{proof}

As a by-product of Theorem~\ref{easy}, we obtain a new asymptotic enumeration formula for
sparse simple uniform hypergraphs with given degrees, 
generalising~\cite[Theorem 1.1]{BG}
(restated earlier as Theorem~\ref{th:BG}, for ease of comparison):
the new formula allows $r$ to grow slowly with $n$, 
whereas~Theorem~\ref{th:BG} is only valid for fixed $r\geq 3$. 
(The two results match when $r$ is constant.)

\begin{corollary}
For $n\geq 3$, let $r=r(n)\geq 3$
be an integer and let $\kvec=\kvec(n) =  (k_1,\ldots, k_n)$ be a vector of
nonnegative integers, where each $k_j=k_j(n)$ may depend on $n$. Let $M = M(n) = \sum_{j=1}^n k_j$
for all $n\geq 3$, and suppose that the set
\[ \mathcal{I} = \{ n\geq 3 \mid r(n) \text{ divides } M(n)\}\]
is infinite. 
Suppose that $M\to\infty$ and $r^4\kmax^3 = o(M)$ as $n$ tends
to infinity along elements of $\mathcal{I}$. Then
\[ |\mathcal{H}_r(\kvec)| = \frac{M!}{(M/r)!\, (r!)^{M/r}\,
 \prod_{i=1}^n k_i!}\,\,
   \exp\left( -\frac{(r-1)M_2}{2M} + O(r^4\,\kmax^3/M)\right).
\]
\label{simple-corollary}
\end{corollary}

\begin{proof}
As noted earlier, the argument leading to (\ref{eq1}) is
still valid when $r^4\kmax^3 = o(M)$.
Since $r\geq 3$, it follows from (\ref{eq1}) that
$|\Brkzero|/|\Brk| = 1 + O(r^4\kmax^3/M)$.
Combining this with (\ref{base}) and (\ref{step1})
completes the proof.
\end{proof}

\section{Double links}\label{s:doubles}

For nonnegative integers $d$, let $\mathcal{C}_{d}$
be the set of bipartite graphs in $\Brkplus$ which 
contain precisely $d$ 4-cycles. 
(The corresponding hypergraph has exactly $d$ double links.)
The sets $\mathcal{C}_d$ partition $\Brkplus$, 
and so
\begin{equation}
\label{partition}
 |\Brkplus| = \sum_{d=0}^{N_2}\, |\mathcal{C}_d|.
\end{equation}
We estimate this sum using a switching operation which we now
define. 

An 8-tuple of distinct vertices
$T= (u_1,u_2,w_1,w_2,f_1,f_2,g_1,g_2)$
is called \emph{suitable} if
\[ u_1,u_2,w_1,w_2\in \{ v_1,\ldots, v_n\}\,\,  \text{ and } \,\,
f_1,f_2,g_1,g_2\in \{ e_1,\ldots, e_{M/r}\}.\] 
A \emph{d-switching} from $B\in\mathcal{C}_d$ is described by a suitable
8-tuple $T$ of vertices of $B$ such that
\begin{itemize}
\item $B$ has a 4-cycle on $\{ u_1,u_2\}\cup\{f_1,f_2\}$,
\item $w_1g_1$ and $w_2g_2$ are edges in $B$.
\end{itemize}
The corresponding d-switching produces a new bipartite graph $B'$
with the same vertex set as $B$ and with edge set
\begin{equation}
\label{Bdef}
 E(B') = \left(E(B) \setminus \{ u_1f_1,\, u_2f_2,\, w_1g_1,\, w_2g_2\}\right)
  \cup \{ u_1g_1,\, u_2g_2,\, w_1f_1,\, w_2f_2\}.
\end{equation}
The d-switching operation is illustrated in Figure~\ref{f:d-switch-bipart}
below.  
(Note that in the hypergraph setting, the d-switching
replaces the four edges $f_1,f_2,g_1,g_2$ of the original hypergraph
with the edges $f_1',f_2', g_1', g_2'$ defined by
\[ f_j' = \left( f_j\setminus \{ u_j\}\right) \cup\{ w_j\},\quad
   g_j' = \left( g_j\setminus \{ w_j\}\right) \cup \{u_j\}\]
for $j=1,2$.)

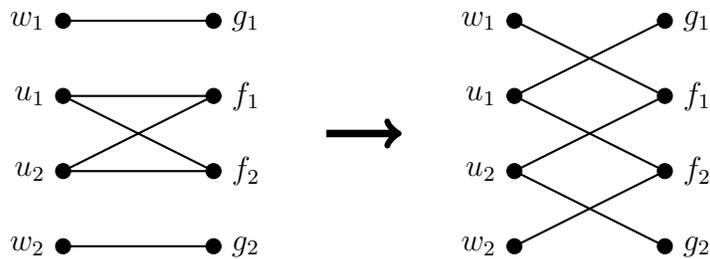
\begin{figure}[ht]
\begin{center}
\begin{tikzpicture}
\draw [fill] (1,0) circle (0.1);
\node [left] at (0.9,3) {$w_1$};
\draw [fill] (1,1) circle (0.1);
\node [left] at (0.9,2) {$u_1$};
\draw [fill] (1,2) circle (0.1);
\node [left] at (0.9,1) {$u_2$};
\draw [fill] (1,3) circle (0.1);
\node [left] at (0.9,0) {$w_2$};
\draw [fill] (3,0) circle (0.1);
\node [right] at (3.1,3) {$g_1$};
\draw [fill] (3,1) circle (0.1);
\node [right] at (3.1,2) {$f_1$};
\draw [fill] (3,2) circle (0.1);
\node [right] at (3.1,1) {$f_2$};
\draw [fill] (3,3) circle (0.1);
\node [right] at (3.1,0) {$g_2$};
\draw [-,thick] (1,3) -- (3,3);
\draw [-,thick] (1,2) -- (3,2) -- (1,1) -- (3,1) -- (1,2);
\draw [-,thick] (1,0) -- (3,0);
\draw [->,line width=1mm] (4.5,1.5) -- (5.5,1.5);
\draw [fill] (7,0) circle (0.1);
\node [left] at (6.9,3) {$w_1$};
\draw [fill] (7,1) circle (0.1);
\node [left] at (6.9,2) {$u_1$};
\draw [fill] (7,2) circle (0.1);
\node [left] at (6.9,1) {$u_2$};
\draw [fill] (7,3) circle (0.1);
\node [left] at (6.9,0) {$w_2$};
\draw [fill] (9,0) circle (0.1);
\node [right] at (9.1,3) {$g_1$};
\draw [fill] (9,1) circle (0.1);
\node [right] at (9.1,2) {$f_1$};
\draw [fill] (9,2) circle (0.1);
\node [right] at (9.1,1) {$f_2$};
\draw [fill] (9,3) circle (0.1);
\node [right] at (9.1,0) {$g_2$};
\draw [-,thick] (7,3) -- (9,2) -- (7,1) -- (9,0);
\draw [-,thick] (9,3) -- (7,2) -- (9,1) -- (7,0);
\end{tikzpicture}
\caption{A d-switching}
\label{f:d-switch-bipart}
\end{center}
\end{figure}

We say that a d-switching from $B\in\mathcal{C}_d$
specified by the (suitable) 8-tuple $T$
is \emph{legal} if the resulting bipartite graph $B'$ belongs to 
$\mathcal{C}_{d-1}$, and otherwise we say that the switching is
\emph{illegal}.

Let $\dist_{\widehat{B}}(x,y)$ denote the length of the shortest path
from $x$ to $y$ in a bipartite graph $\widehat{B}$.

\begin{lemma}
Suppose that $d\leq N_2$ is a positive integer and $B\in\mathcal{C}_d$.
With notation as above, if the d-switching from $B$ specified by $T$ 
is illegal then at least one of the following conditions must hold:
\begin{enumerate}
\item[\emph{(I)}] At least one of $g_1$ or $g_2$ belongs to a 4-cycle in $B$. 
\item[\emph{(II)}] For some $j\in \{ 1,2\}$, either 
$\dist_B(u_j,g_j)\leq 3$ or $\dist_B(w_j,f_j)\leq 3$. 
\item[\emph{(III)}] $\dist_B(g_1,g_2)= 2$.
\end{enumerate}
\label{forwardgood}
\end{lemma}

\begin{proof}
Fix $B\in\mathcal{C}_{d}$ and let $T$ describe a d-switching from $B$
such that the resulting bipartite graph $B'$ does not belong to 
$\mathcal{C}_{d-1}$.
First, suppose that $B'\in\Brkplus$ but that $B'$ contains 
at most $d-2$ 4-cycles.  Then the d-switching has destroyed
more than one 4-cycle, which implies that (I) holds.

Next, suppose that there exists a 4-cycle in $B'$ which contains
an edge of $E(B')\setminus E(B)$.  Call these \emph{new edges}.
Such a 4-cycle has been (accidently) created by the d-switching. 

First suppose that $B'$ contains a 4-cycle which involves
precisely one new edge.
If $B'$ contains a 4-cycle which involves the
edge $u_jg_j$ (for some $j\in \{ 1,2\}$)
 but does not involve any other new edge then
$\dist_B(u_j,g_j)\leq 3$, which implies that (II) holds.
Similarly, if $B'$ contains a 4-cycle which contains the
edge $w_jf_j$ for some $j\in\{1,2\}$, but contains no other new edge,
then $\dist_B(w_j,f_j)\leq 3$,
which again implies that (II) holds.
Now suppose that there are 4-cycles in $B'$ which contain at least
two new edges.
If a 4-cycle in $B'$ contains both
$w_1f_1$ and $w_2f_2$ then $w_1f_2\in E(B)$, which implies that
(II) holds. 
If a 4-cycle in $B'$ contains both $u_1g_1$ and $u_2g_2$
then $u_2g_1\in E(B)$, so (II) holds.
No 4-cycle in $B'$ can
contain both $u_jg_j$ and $w_jf_j$ for some $j\in \{1,2\}$, 
since the edge $u_j f_j$ is not present in $B'$.
Next, suppose that there is a 4-cycle in $B'$ which contains 
both $u_jg_j$ and $w_\ell f_\ell$, where $\{ j,\ell\} = \{1,2\}$.
Then $w_\ell g_j\in E(B)$, so (III) holds. 

The arguments given above cover the case that
$B'\in\Brkplus$ but that $B'$ has
strictly more than $d-1$ 4-cycles, since the d-switching must
have introduced at least one new 4-cycle.  Furthermore,
it follows from properties (i)--(v) that if
$B'\not\in \Brkplus$ then there must be a 4-cycle
in $B'$ which contains a new edge.
(For example, if (i) fails for $B'$
then there is a copy of $K_{3,2}$ involving at least one new edge,
but then that new edge is contained in at least one 4-cycle in $B'$.)
Hence this case is also covered by the above arguments, completing
the proof.
\end{proof}

A \emph{reverse d-switching} 
is the reverse of a d-switching.
A reverse d-switching from a bipartite graph
$B'\in\mathcal{C}_{d-1}$ 
is described by a suitable 8-tuple $T$ of vertices
such that
\[ u_1g_1,\,\, u_2g_2,\,\, u_1f_2,\,\, u_2f_1,\,\, w_1f_1,\, 
  w_2f_2\]
are all edges of $B'$.  The reverse d-switching produces
the bipartite graph $B$ defined by (\ref{Bdef}).
This operation is depicted in Figure~\ref{f:d-switch-bipart} by
following the arrow in reverse.

Given $B'\in\mathcal{C}_{d-1}$, we say that a reverse
d-switching from $B'$ specified by the (suitable) 8-tuple $T$ 
is \emph{legal} if the resulting bipartite graph $B$
belongs to $\mathcal{C}_d$, and otherwise we say that the switching
is \emph{illegal}.  

The proof of the following is very similar to the proof of 
Lemma~\ref{forwardgood}, but for completeness we give it in full.

\begin{lemma}
Suppose that $d\leq N_2$ is a positive integer and $B'\in\mathcal{C}_{d-1}$.
With notation as above, if the reverse switching from $B'$
specified by $T$ is illegal then at
least one of the following conditions must hold:
\begin{enumerate}
\item[\emph{(I${}'$)}] At least one of $u_1,u_2,f_1,f_2,g_1,g_2$ belongs to a 4-cycle in $B'$. 
\item[\emph{(II${}'$)}] 
For some $j\in \{ 1,2\}$, either
$\dist_{B'}(u_j,f_j)\leq 3$ or $\dist_{B'}(w_j,g_j)\leq 3$.
\end{enumerate}
\label{reversegood}
\end{lemma}

\begin{proof}
Fix $B'\in\mathcal{C}_{d-1}$ and let $T$ describe a reverse d-switching from $B'$
such that the resulting bipartite graph $B$ does not belong to $\mathcal{C}_d$.
First, suppose that $B\in\Brkplus$ but that $B$ contains
at most $d-1$ 4-cycles.  Then the reverse d-switching has
destroyed at least one 4-cycle, so (I$'$) holds.

Clearly any new 4-cycle in $B$ created by the reverse d-switching
must contain at least one edge of $E(B)\setminus E(B')$. (Again, we
call these \emph{new edges}.)
Of course, the reverse d-switching is designed to create a 
new 4-cycle involving the edges $u_1f_1,\, u_1f_2$, but here we
are only interested in other 4-cycles which may ``accidently'' be
created by the reverse d-switching.

If a 4-cycle in $B$ contains precisely one new edge then 
(II$'$) holds in $B'$.   Next suppose that a new 4-cycle
in $B$ contains at least two new edges.
If any new 4-cycle in $B$ contains both $w_1g_1$ and $w_2g_2$
then $w_1g_2\in E(B')$. This gives a 4-cycle in $B'$ involving $u_2$,
and so (II$'$) holds.  No new 4-cycle in $B$ can contain both 
$u_jf_j$ and $w_j g_j$, 
since $u_jg_j\not\in E(B)$, for any $j \in\{ 1,2\}$.
Next, if $u_jf_j$ and $w_\ell g_\ell$ belong to a 4-cycle in $B$,
where $\{ j,\ell\} = \{1,2\}$, then
$u_j g_\ell\in E(B')$ and (II$'$) holds.

The above argument covers the possibility that $B'\in \Brkplus$ 
but that $B$ contains more than $d$ 4-cycles.
Now suppose that $B$ contains precisely $d$ 4-cycles but $B\not\in\Brkplus$.
Note that property (v) holds, by our assumption on $d$.
If property (i) or (ii) fails for $B$ then at least one additional
4-cycle has been created by the reverse d-switching, which was
covered by the above argument.
If property (iii) fails for $B$ then either $f_1$ or $f_2$
must belong to a 4-cycle in $B'$, 
while if property (iv) fails for $B$ then either $u_1$ or $u_2$
must belong to a 4-cycle in $B'$. Thus (I$'$) holds in both cases,
completing the proof.
\end{proof}

We will analyse d-switchings to obtain an asymptotic expression
for $|\mathcal{C}_d|/|\mathcal{C}_{d-1}|$, and then combine these
to find an expression for $|\mathcal{L}_r(\kvec)| = |\mathcal{C}_0|/(M/r)!$,
which is the quantity of interest.
First we analyse one d-switching.

\begin{lemma}
Assume that the conditions of Theorem~\ref{main} hold.
Let $d'$ be the first value of $d\leq N_2$ such that $\mathcal{C}_d=\emptyset$,
or $d' = N_2+1$ if no such value exists.  If $d\in \{ 1,\ldots, d'-1\}$ then
\[ |\mathcal{C}_{d}|
    = |\mathcal{C}_{d-1}| \,
\frac{(r-1)^2 M_2^2}{4dM^2}\, 
   \left( 1 + O\left(\frac{d\kmax(\kmax + r)+r^2\kmax^3}{M_2}\right)\right).
\]
\label{d-switch-easy}
\end{lemma}

\begin{proof}
Fix $d\in \{ 1,\ldots, d' - 1\}$ and let $B\in\mathcal{C}_{d}$ be given.  
Let $\mathcal{S}$ be the set of all suitable 8-tuples $T$ such that
\begin{itemize}
\item $B$ contains a 4-cycle on $\{ u_1,u_2,f_1,f_2\}$,
\item the edges $w_1g_1$, $w_2g_2$ belong to $B$, and
\item no 4-cycle in $B$ contains $g_1$ or $g_2$.
\end{itemize}
Then $\mathcal{S}$ contains every 8-tuple which defines a legal d-switching
from $B$, so $|\mathcal{S}|$ is an upper bound for the number of legal
d-switchings from $B$.  There are precisely $d$
4-cycles, and 4 ways to order the vertices $(u_1,u_2,f_1,f_2)$.
For an upper bound, there are at most $M^2$ ways to choose the edges
$w_1g_1$, $w_2g_2$ in order, giving $|\mathcal{S}|\leq 4d M^2$.
To give a lower bound, we must ensure that all vertices are distinct
and that $g_1$ and $g_2$ are not contained in any 4-cycle.
Given $(u_1,u_2,f_1,f_2)$, there are at least
\[ (M - (2rd + 2\kmax))(M - ((2d+1)r + 3\kmax))\]
good choices for $(w_1,w_2,g_1,g_2)$. Hence
\[
 |\mathcal{S}| = 4d M^2\left( 1 +O\left(\frac{rd + \kmax}{M}\right)\right).
\]
We now find obtain an upper bound for the number of 8-tuples in $\mathcal{S}$
which give rise to illegal d-switchings from $B$, and subtract this value
from $|\mathcal{S}|$.  By Lemma~\ref{forwardgood} it suffices to find an upper
bound for the number of 8-tuples in $\mathcal{S}$ which satisfy Condition (II)
or Condition (III).  Observe that no 8-tuple in $\mathcal{S}$ satisfies 
Condition (I), by definition of $\mathcal{S}$.
For Condition (II), there are 
$O(dr\kmax M)$ 8-tuples in $\mathcal{S}$ such that 
an edge exists in $B$ from $u_1$ to $g_j$ or
from $w_j$ to $f_j$, for some $j\in \{1,2\}$.  Similarly, there are 
$O(dr^2 \kmax^2 M)$ 8-tuples for which $\dist_B(u_1,g_j)=3$ or 
$\dist_B(w_j,f_j)=3$, for some $j\in\{1,2\}$.
Hence Condition (II) fails for $O(dr^2\kmax^2 M)$ 8-tuples in $\mathcal{S}$.

Similarly, there are $O(d r^2 \kmax M)$ 8-tuples in $\mathcal{S}$ which
satisfy Condition (III).  Combining these contributions, it follows that
there are 
\begin{equation}
\label{nr-forward}
  4dM^2\left(1 + O\left(\frac{rd+r^2\kmax^2 }{M}\right)\right)
\end{equation}
suitable 8-tuples which give a legal d-switching from $B$.

Next, suppose that $B'\in\mathcal{C}_{d-1}$ (and note that $\mathcal{C}_{d-1}$
is nonempty, by definition of $d'$).  Let $\mathcal{S}'$ be the set of
all suitable 8-tuples such that
\begin{itemize}
\item $u_1 g_1$, $u_2  g_2$, $u_1 f_2$, $u_2 f_1$, $w_1 f_1$, $w_2 f_2$
are all edges of $B'$, and
\item no 4-cycle in $B'$ contains a vertex from $\{ u_1,u_2,f_1,f_2,g_1,g_2\}$.
\end{itemize}
Again, $\mathcal{S}'$ contains every 8-tuple which describes a legal reverse
d-switching from $B'$.  Hence there are at most $|\mathcal{S}'|$ legal
reverse d-switchings from $B'$.  There are at most
$M_2$ ways to choose $(u_1,f_2,g_1)$ and at most $M_2$ ways to choose
$(u_2,f_1,g_2)$, and then at most $(r-1)^2$ ways to choose $(w_1,w_2)$.
Therefore $|\mathcal{S}'| \leq (r-1)^2 M_2^2$.  

For a lower bound, we
must ensure that all vertices are distinct and that we avoid choosing
$u_1,u_2,f_1,f_2,g_1,g_2$ from a 4-cycle. 
We can choose $(u_1,f_2,g_1)$, avoiding vertices contained in 4-cycles, 
in at least
$M_2 - 2(d-1)\kmax(\kmax + 2r)$ ways.  There are still precisely $(r-1)$
choices for $w_2$ from among all neighbours of $f_2$ other than $u_1$.
Next, there are at least
\[ M_2 - 2(d-1)\kmax(\kmax + 2r) - 3\kmax^2 - 4r\kmax - 2r\kmax^2\]
ways to choose $(u_2,f_1,g_2)$ avoiding vertices contained in 4-cycles 
and avoiding those vertices already chosen, such that $f_1$ is not
a neighbour of $u_1$ or $w_2$ in $B'$.  This choice of $f_1$ ensures
that all $r-1$ neighbours of $f_1$ other than $u_2$
are also distinct from $\{ u_1,w_2\}$, so there are still $r-1$
choices for $w_2$.
It follows that
\[ |\mathcal{S}'| = (r-1)^2 M_2^2\left(1 + 
              O\left(\frac{d\kmax(\kmax + r) + r\kmax^2}{M_2}\right)\right).
\]
Now we calculate an upper bound for the number of 8-tuples in $\mathcal{S}'$
which give an illegal reverse d-switching from $B'$. By Lemma~\ref{reversegood},
it suffices to find an upper bound for the number of 8-tuples in $\mathcal{S}'$
which satisfy Condition (II$'$). (Note that no element of $\mathcal{S}'$
can satisfy Condition (I$'$), by definition of $\mathcal{S}'$.)
There are $O(r^3\kmax^2 M_2)$ elements of $\mathcal{S}'$ such that
there is an edge from $w_j$ to $g_j$
or an edge from $u_j$ to $f_j$, for some $j\in \{ 1,2\}$.
Similarly, the number of 8-tuples in $\mathcal{S}'$
with $\dist_{B'}(u_j,f_j)=3$ or
with $\dist_{B'}(w_j,g_j)=3$ for some $j\in \{ 1,2\}$ is
$O(r^4\kmax^3 M_2)$.
Hence Condition (II${}'$) fails for $O(r^4 \kmax^3 M_2)$ 8-tuples,
which (together with the upper bound on $|\mathcal{S}'|$) 
implies that
the number of legal reverse d-switchings from $B'$ is
\begin{equation}
\label{nr-reverse}
 (r-1)^2 M_2^2\left(1 + O\left(\frac{d\kmax(\kmax + r) + r^2\kmax^3}{M_2}\right)\right).
\end{equation}
Comparing the error terms from (\ref{nr-forward})  and (\ref{nr-reverse}),
we see that
the error term from the reverse d-switchings is largest, since $1/M \leq \kmax/M_2$.
Taking the ratio of (\ref{nr-forward}) and (\ref{nr-reverse}) completes the proof. 
\end{proof}

We can now prove our main result.  The proof
is similar to those in related enumeration results
such as~\cite{sissy}.  We present the proof in full 
in order to demonstrate how the factors of $r$ arise in the error bounds
(since previous results only dealt with $r=2$, or assumed that $r$ was constant).
The following summation lemma from~\cite{GMW} will be needed.
(The statement has been adapted slightly from that given in~\cite{GMW}, 
without affecting the proof given there.)

\begin{lemma}[{\cite[Corollary 4.5]{GMW}}]\label{sumcor2}
Let $N\geq 2$ be an integer and, for $1\leq i\leq N$, let real
numbers $A(i)$, $C(i)$ be given such that $A(i)\geq 0$ and
$A(i)-(i-1)C(i) \ge 0$.
Define 
$A_1 = \min_{i=1,\ldots, N} A(i)$, $A_2 = \max_{i=1,\ldots, N} A(i)$,
$C_1 = \min_{i=1,\ldots, N} C(i)$ and $C_2=\max_{i=1,\ldots, N} C(i)$.
Suppose that there exists a real number $\hat{c}$ with 
$0<\hat{c} < \tfrac{1}{3}$ such that 
$\max\{ A_2/N,\, \abs{C_1},\, \abs{C_2}\} \leq \hat{c}$.
Define $n_0,\ldots ,n_N$ by $n_0=1$ and
\[ n_i = \frac{1}{i}\left(A(i)-(i-1)C(i)\right)\, n_{i-1} \]
for $1\leq i\leq N$.  Then
\[ \varSigma_1 \leq \sum_{i=0}^N n_i\leq \varSigma_2, \]
where
\begin{align*}
\varSigma_1 &= \exp\left( A_1 - \tfrac{1}{2} A_1 C_2 \right)
               - (2e\hat{c})^N,\\
 \varSigma_2 &= \exp\left( A_2 - \tfrac{1}{2} A_2 C_1 +
              \tfrac12 A_2 C_1^2 \right) + (2e\hat{c})^N.
\end{align*}
\end{lemma}

\bigskip

\begin{proof}[Proof of Theorem~\ref{main}]\
First we prove that
\begin{equation}
 \sum_{d=0}^{N_2} |\mathcal{C}_d| =  |\mathcal{C}_0|\, 
   \exp\left( \frac{(r-1)^2M_2^2}{4M^2} + 
    O\left(\frac{r^{4}\kmax^4(\kmax + r)}{M}\right)\right). 
\label{aim}
\end{equation}
Let $d'$ be the first value of $d\leq N_2$ for which $\mathcal{C}_{d}=\emptyset$,
or $d=N_2+1$ if no such value of $d$ exists.  
We saw in Lemma~\ref{d-switch-easy} that any $B\in\mathcal{C}_d$
can be converted to some $B'\in\mathcal{C}_{d-1}$ using a
d-switching.
Hence $\mathcal{C}_{d}=\emptyset$ for $d'\leq d\leq N_2$.
In particular, (\ref{aim}) holds if $\mathcal{C}_{0}=\emptyset$,
so we assume that $d'\geq 1$.

By Lemma~\ref{d-switch-easy}, there is some uniformly bounded function 
$\alpha_d$ such that 
\begin{equation}\label{drat}
\frac{|\mathcal{C}_{d}|}{|\mathcal{C}_{0}|}
  =  
\frac{1}{d}\, \frac{|\mathcal{C}_{d-1}|}{|\mathcal{C}_{0}|}\,
 \left(A(d) -  (d-1) C(d) \right)
\end{equation}
for $1\leq d\leq N_2$, where
\[
  A(d) = \frac{(r-1)^2M_2^2- \alpha_d\, r^4 \kmax^3 \, M_2}{4M^2},
  \quad C(d) = \frac{\alpha_d \, r^2 \kmax (\kmax + r)\, M_2}{4M^2}
\]
for $1\leq d < d'$, and $A(d) = C(d) = 0$ for $d'\leq d \leq N_2$.

We wish to apply Lemma~\ref{sumcor2}.
It is clear that $A(d)-(d-1)C(d)\geq 0$, from (\ref{drat}) if
$1\leq d < d'$ or by definition, if $d'\leq d\leq N_2$.
If $\alpha_d \geq 0$ then $A(d) \geq A(d) - (d-1)C(d) \geq 0$ by
(\ref{drat}), while if $\alpha_d < 0$ then $A(d)$ is nonnegative by
definition.
Now define $A_1, A_2, C_1, C_2$ by taking the minimum and maximum
of $A(d)$ and $C(d)$ over $1\leq d \leq N_2$. 
Let $A\in [A_1,A_2]$ and $C\in [C_1,C_2]$
and set $\hat{c}=\frac{1}{20}$.   Since $A = (r-1)^2 M_2^2/4M^2 + o(1)$
and $C = o(1)$, we have that $\max\{ A/N_2,\, |C|\}\leq \hat{c}$
for $M$ sufficiently large, by the definition of $N_2$.  
Hence Lemma~\ref{sumcor2} applies and gives an upper bound
\[ \sum_{d = 0}^{N_2} \frac{|\mathcal{C}_{d}|}
   {|\mathcal{C}_{0}|} 
 \leq \exp\left( \frac{(r-1)^2M_2^2}{4M^2} + 
    O\left(\frac{r^4\kmax^4(\kmax + r) }{M}\right)\right) 
        + O\bigl( (e/10)^{N_2}\bigr).\]
Since 
$(e/10)^{N_2}\leq (e/10)^{3\log M} \leq M^{-1}$, this gives
\[ \sum_{d = 0}^{N_2} \frac{|\mathcal{C}_{d}|}
   {|\mathcal{C}_{0}|} 
  \leq \exp\left( \frac{(r-1)^2M_2^2}{4M^2} + 
    O\left(\frac{r^{4}\kmax^4(\kmax + r)}{M}\right)\right) .
\]
In the case that $d'=N_2 + 1$, the lower bound given by Lemma~\ref{sumcor2}
is the same within the stated error term, which establishes (\ref{aim}) in this case.
 
This leaves the case that $1\leq d' \leq N_2$.   Considering the
analysis of the reverse switchings from Lemma~\ref{d-switch-easy},
this case can only arise if
\[ M_2 = O(d'\kmax(\kmax + r) + r^2\kmax^3) = 
      O\left(\kmax(\kmax + r)(r^2\kmax^2 + \log M)\right).\]
But then
\[ \frac{(r-1)^2M_2^2}{4M^2} = 
    O\left(\frac{r^2\kmax^2(\kmax + r)^2(r^2\kmax^2 + \log M)^2}{M^2} \right) 
    = O\left(\frac{r^4\kmax^4}{M}\right),
\]
so the trivial lower bound of 1 matches the upper bound
within the error term.  Hence (\ref{aim}) also holds when $1\leq d'\leq N_2$.  

Therefore (\ref{aim}) holds in both cases. Combining (\ref{base}),
Theorem~\ref{easy} and (\ref{partition}) gives
\begin{align*}
& |\mathcal{L}_r(\kvec)| \\
 &= \frac{|\mathcal{C}_0|}{(M/r)!}\\
        &= \frac{|\Brkplus|}{(M/r)!}\, 
   \exp\left( -\frac{(r-1)^2M_2^2}{4M^2} + 
    O\left(\frac{r^{4}\kmax^4(\kmax + r)}{M}\right)\right)\\
        &= \frac{M!}{(M/r)!\, (r!)^{M/r}\, \prod_{j=1}^n k_j!}\, 
   \exp\left( -\frac{(r-1)M_2}{2M} - \frac{(r-1)^2M_2^2}{4M^2} + 
    O\left(\frac{r^{4}\kmax^4(\kmax + r)}{M}\right)\right),
\end{align*}
completing the proof.
\end{proof}

As a corollary of Theorem~\ref{easy} and (\ref{aim}) we obtain 
the following result regarding the girth of bipartite graphs.

\begin{corollary}
Under the conditions of Theorem~\ref{main}, the probability that a randomly 
chosen element of $\Brk$ has no 4-cycle, and hence has girth at least 6, is
\[ \frac{|\mathcal{C}_0|}{|\Brk|} = 
  \exp\left( - \frac{(r-1)^2\, M_2^2}{4 M^2} + O\left(\frac{r^4\kmax^4(\kmax + r)}{M}\right)\right).
\]
\label{4cycle}
\end{corollary}

McKay, Wormald and Wysocka~\cite[Corollary 3]{MWW} proved
the following: 
if $(d-1)^7 = o(n)$ as $n\to\infty$ along the positive even integers then
that the probability that
a random $d$-regular bipartite graph on $n$ vertices
has girth greater than $g$ is
\[ \exp\left( - \frac{(d-1)^4}{4} + o(1)\right).\]
 (The conclusion was known much earlier for constant $d$; see~\cite{Wthesis}.) 
Corollary~\ref{4cycle} can be seen as a generalisation 
of the $g=4$ case of~\cite[Corollary 3]{MWW}
to bipartite graphs which are irregular on one side of the 
vertex bipartition, and are sufficiently sparse.
When the bipartite graph is $d$-regular (with $\kmax = r = d$),
the condition of Corollary~\ref{4cycle} becomes $d^8 = o(n)$,
which is slightly more restrictive than that of~\cite{MWW}.

\subsection*{Acknowledgements}
We are very grateful to Brendan McKay for suggesting 
that use of the incidence
matrix of the hypergraph would simplify the calculations.
We thank the referee for their helpful comments.

\end{document}